\documentclass[11pt]{amsart}
\usepackage[utf8]{inputenc}
\usepackage{a4wide}
\usepackage{amsmath,eucal, mathtools}
\usepackage{xcolor}
\usepackage{bbm,stix}
\usepackage{enumitem}
\usepackage{graphics} 
\usepackage{epsfig} 
\usepackage{graphicx,subfig}                     
\usepackage{wrapfig}
\usepackage{textcomp}
\usepackage{tikz}
\usepackage{pgfplots}
\usepackage{dsfont}
\pgfplotsset{compat=1.16}
\usepackage{comment}
\usepackage{multicol} 
\usepackage{mathrsfs}
\usepackage{epstopdf}
\usepackage{cancel}
\usepackage[colorlinks=true,allcolors=black]{hyperref}
\usepackage{bm}                        
\usepackage[]{appendix}  
\usepackage{booktabs} 
\usepackage{graphicx,subfig}                                                       
\usepackage{float} 
\usepackage{wrapfig}              

\newtheorem{defn}{Definition}[section]
\newtheorem{thm}{Theorem}[section]
\newtheorem{prop}{Proposition}[section]
\newtheorem{lem}{Lemma}[section]
\newtheorem{cor}{Corollary}[section]

\newtheorem{ass}{Assumptions}[section]
\newtheorem{rem}{Remark}[section]

\DeclareMathOperator*{\argmin}{argmin}

\newcommand{\N}{\mathbb{N}}
\newcommand{\R}{\mathbb{R}}
\newcommand{\Rn}{{\mathbb{R}^{n}}}
\newcommand{\Rnu}{{\mathbb{R}^{d}}}
\newcommand{\Rnn}{{\mathbb{R}^{2n}}}

\newcommand{\mpd}{{\mathscr{P}_{2}(\Rn)}}
\newcommand{\mpdu}{{\mathscr{P}_{2}(\Rnu)}}

\newcommand{\mpdm}[1]{{\mathscr{P}^{#1}_{2}(\Rnu)}}
\newcommand{\mpdmp}[1]{{\mathscr{P}^{#1}_{p}(\Rnu)}}
\newcommand{\pd}[1]{\mathscr{P}_{#1}(\Rn)}

\newcommand{\mF}{{\mathscr{F}}}

\newcommand{\dt}{{\mathrm{d}t}}

\newcommand{\dd}{{\mathrm{d}}}
\newcommand{\1}{\mathbb{1}_{S_0}}
\newcommand{\F}{\mathscr{F}_A}

\definecolor{brightmaroon}{rgb}{0.76, 0.13, 0.28}
\definecolor{ceruleanblue}{rgb}{0.16, 0.32, 0.75}

\title[Weak Solutions to a Constrained Aggregation-Diffusion-Reaction Model for MS]{Existence of Weak Solutions to a Constrained Aggregation-Diffusion-Reaction Model for Multiple Sclerosis}
\author{S. Fagioli  \and M. Kamath Katapady}
\address{ Simone Fagioli - DISIM - Department of Information Engineering, Computer Science and Mathematics, University of L'Aquila, Via Vetoio 1 (Coppito)
67100 L'Aquila (AQ) - Italy}
\email{simone.fagioli@univaq.it}
\address{ Megha Kamath Katapady - DISIM - Department of Information Engineering, Computer Science and Mathematics, University of L'Aquila, Via Vetoio 1 (Coppito)
67100 L'Aquila (AQ) - Italy}
\email{megha.kamathkatapady@graduate.univaq.it}

\keywords{multiple sclerosis; non-linear diffusion; non-local interaction; splitting schemes; variational schemes} 
\subjclass[2020]{35K65;35Q92;35A15;45K05;92C17}
\begin{document}

\begin{abstract}
We establish an existence result for weak solutions to an aggregation-diffusion-reaction equation with a constraint, arising in the modelling of multiple sclerosis. The model is derived from a general chemotaxis-type framework and describes the time evolution of the density of activated macrophages, which is subject to attraction by oligodendrocytes. The latter are governed by a constraint equation. The proof relies on a variational splitting scheme that isolates the transport (aggregation-diffusion) and reaction contributions. The structure of the constraint makes it possible to recover the oligodendrocyte density as the limit of a sequence of characteristic functions.
\end{abstract}

\maketitle

\section{Introduction}
We are interested in studying the well-posedness of the following constrained aggregation-diffusion-reaction equation arising in the modelling of Multiple Sclerosis disease
\begin{equation}\label{maineq}
\begin{cases}
     \partial_t\rho = \nabla\cdot \left( \rho\nabla\Phi'(\rho) -\chi\rho\nabla K\ast (\rho + \beta)\right)+M(\rho)\\
     h(\rho)\left(1-\beta\right)=0,
\end{cases}
\end{equation}
for $x\in\Rnu$ and $t>0$. In \eqref{maineq} the unknown is the couple $(\rho,\beta)$, where $\rho$ represents a density describing activated macrophages and $\beta$ is a density describing destroyed oligodendrocytes, see below for a breif biological description. In the system above, $\Phi$, $M$ and $h$ are given functions modelling non-linear diffusion, reaction, and saturation effects respectively, $K$ is the Bessel interaction kernel, while $\chi>0$ is the chemotaxis coefficient. 

Multiple sclerosis (MS) is a chronic inflammatory disease of the central nervous system, characterized by immune-mediated demyelination, macrophages activation, oligodendrocyte damage, and axonal degeneration. The pathological hallmark of the disease is the formation of focal lesions in the brain and spinal cord, whose spatial organization ranges from classical plaques to more complex concentric structures such as those observed in Baló's concentric sclerosis, see \cite{Balo,Barnett09,Lass18,Lucchinettietal00,marik}. These heterogeneous patterns reflect the intricate interplay between immune cells, chemical mediators, and neural tissue, and highlight the importance of spatially distributed mechanisms in the progression of the disease.

Mathematical modelling has emerged as a valuable tool for understanding the mechanisms underlying lesion formation and evolution in MS. Early modelling approaches mainly relied on systems of ordinary differential equations to describe the temporal evolution of immune cell populations and inflammatory mediators, as well as the balance between myelin damage and repair \cite{ElettrebyAhmed2020,Jenner2025MSModel,Kotelnikova2017,MoFr21}. These compartmental ODE models have provided insight into disease progression and relapse-remission dynamics by capturing interactions between immune activation, demyelination, and neurodegeneration. However, by construction, they neglect spatial effects and therefore cannot account for lesion morphology, spatial heterogeneity, or pattern formation observed in clinical imaging.

To overcome this limitation, a growing body of literature has focused on spatially extended models based on partial differential equations. In particular, reaction-diffusion-chemotaxis systems have proven effective in describing the migration of immune cells driven by chemical gradients and their interaction with damaged tissue. In these models, macrophages or activated immune cells undergo random diffusion combined with directed chemotactic movement toward chemoattractants released by apoptotic oligodendrocytes or inflammatory cytokines. This modelling framework naturally explains the emergence of spatial patterns and has been shown to reproduce lesion-like structures observed in clinical imaging.

A seminal contribution in this direction is the chemotaxis-based model introduced in \cite{CalvezKhonsari2008,KhonsariCalvez2008} to describe demyelination patterns in MS, where the interaction between macrophages, chemoattractants, and myelin degradation leads to diffusion-driven instabilities and Turing-type pattern formation, see also \cite{Lombardo2017}. Subsequent analytical and numerical studies have deepened the understanding of pattern dynamics and instability mechanisms in this model: wavefront propagation and the emergence of concentric rings reminiscent of Balo’s sclerosis were studied through weakly nonlinear analysis in \cite{Barresi2016}; radially symmetric (axisymmetric) solutions and their bifurcations were explored in \cite{Bilotta2019}; and pattern stability through Eckhaus and zigzag instabilities was examined in \cite{Bilotta2018}. These works collectively establish chemotaxis as a key mechanism in the mathematical description of MS lesion dynamics. More recent modifications and extensions - including models that incorporate cytokine-mediated effects on macrophage activation \cite{Gargano2024} - continue to build on this chemotaxis framework and highlight its central role in connecting immune migration with spatial lesion structure.

Furthemore, several extensions of the classical chemotaxis framework have been proposed to enhance biological realism. Models incorporating nonlinear or degenerate diffusion have been introduced to better account for volume-filling effects and the constrained movement of cells in neural tissue \cite{Fagioli2025}. Other studies have included Allee-type effects in macrophage activation to model threshold-driven inflammatory responses and their impact on pattern formation \cite{Bisi2023Allee}. Furthermore, kinetic-based derivations of macroscopic chemotaxis models have recently been proposed, providing a multiscale justification of the governing equations and shedding light on the emergence of two-dimensional spatial patterns \cite{Bisi2025Kinetic}.

These MS-specific models are deeply rooted in the broader mathematical theory of chemotaxis, whose prototype is the classical Keller-Segel system \cite{HillenPainter2009}. In particular, the model introduced in \cite{CalvezKhonsari2008}, and the nonlinear version proposed in \cite{Fagioli2025}, consists of a system of equations governing the evolution of three entities: macrophages $\rho$, cytokines $c$, and apoptotic oligodendrocytes $\beta$, in a bounded domain $\Omega \subset \Rnu$ with smooth boundary,
\begin{equation}\label{system}
\left\{
\begin{aligned}
    \partial_t \rho &= \nabla \cdot \left(\rho \nabla \bigl(\Phi'(\rho) - \chi \nabla c \bigr)\right) + M(\rho),\\
    \tau \partial_t c &= \Delta c - c + \rho + \beta,\\
    \tau \partial_t \beta &= h(\rho) (1 - \beta),
\end{aligned}
\right.
\end{equation}
where $\tau>0$ is a time-scaling parameter. For this system, global well-posedness and boundedness of solutions were proved in \cite{Fagioli2025} for smooth domains with Neumann boundary conditions, by adapting well-known techniques developed in \cite{li2016,tao2008,tao2011}. However, to the best of author knowledge, no analysis of pattern formation or asymptotic behaviour was carried out.

We take a first step in this direction by considering an intermediate asymptotic regime, namely a pseudo-stationary limit $\tau\to 0^+$ in \eqref{system}, formally described in the following. As a first ansatz, we assume that the cytokine dynamics are faster than those of macrophages and, similarly to the classical Keller-Segel framework, we set
\begin{equation*}
    \Delta c - c + \rho + \beta = 0.
\end{equation*}
Considering the problem in the whole space $\Rnu$, classical theory shows that this equation can be solved by convolution with the Bessel kernel $K$,
\begin{equation*}
    c(t,x) = K \ast (\rho + \beta)(t,x),
\end{equation*}
see for instance \cite{adams1999function}. Additional details concerning the Bessel kernel and its properties will be presented later. Thus, the first equation in \eqref{system} can be rephrased as
\begin{equation}\label{firsteq}
     \partial_t \rho = \nabla \cdot \left(\rho \nabla \bigl(\Phi'(\rho) - \chi \nabla K \ast (\rho + \beta) \bigr)\right) + M(\rho).
\end{equation}

Numerical simulations and analytical results in \cite{Fagioli2025,Lombardo2017} suggest that pattern formation is mainly driven by the evolution of the support of activated macrophages. This observation is also supported by the structure of the third equation in \eqref{system}. By additionally assuming that the third equation is at equilibrium, 
\begin{equation}\label{secondeq}
    h(\rho)(1-\beta)=0,
\end{equation}
we formally reduce the system \eqref{system} to \eqref{maineq}. A rigorous justification of the steps described above, and in particular of the reduction of \eqref{system} to \eqref{maineq}, is left to future work and we only focus now on existence of weak solutions to \eqref{maineq}. 

We summarise the strategy we adopt to prove the existence of weak solutions for system \eqref{maineq}. Recalling that the function $h$ in the second equation accounts for saturation effects and, in particular, that we may assume $h(0)=0$, equation \eqref{secondeq} is satisfied either if $\rho=0$ or if $\beta=1$ whenever $\rho\neq 0$. Therefore, the asymptotic behaviour of $\beta$ is roughly that of the characteristic function of the support of $\rho$.

Neglecting for the moment the contribution of the reaction term, we can formally interpret equation \eqref{firsteq} as
\begin{equation}\label{altmaineq}
    \partial_t \rho = \nabla \cdot \left(\rho \nabla\left(\Phi'(\rho)-\chi K\ast \rho - \chi K\ast \mathbb{1}_A\right)\right),
\end{equation}
where $A$ is a time-dependent set, in some sense related to $\operatorname{supp}(\rho)$.  

Equation \eqref{altmaineq} can then be formally formulated as the $2$-Wasserstein gradient flow
\begin{equation*}
    \partial_t \rho = \nabla \cdot \left(\rho \nabla\left(\frac{\delta \F}{\delta \rho}[\rho]\right)\right),
\end{equation*}
associated with the energy functional
\begin{equation*}
    \F[\rho]
    = \int_{\Rnu} \Phi(\rho)
      - \frac{\chi}{2}\int_{\Rnu} \rho\, K\ast \rho
      - \chi \int_{\Rnu} \rho\, K\ast \mathbb{1}_A ,
\end{equation*}

see \cite{ambrosio2008gradient,santambrogio2015optimal}. Despite this intuitive interpretation, making the identification of \eqref{firsteq} with \eqref{altmaineq} rigorous is highly nontrivial, in particular due to the difficulty of properly characterising the set $A$ and its relation to $\operatorname{supp}(\rho)$. However, we exploit this intuition to implement an implicit-explicit JKO scheme, in the spirit of \cite{difrancescofagioli13}, in order to prove the existence of weak solutions to \eqref{maineq} in the case $M=0$; see Section~\ref{sec:JKO} for further details.

The reaction term is then incorporated through a variational splitting scheme in order to construct weak solutions to the full system \eqref{maineq}, see \cite{cafasasc}. More precisely, during the first inner time step we solve the aggregation-diffusion part of the system by means of the implicit-explicit JKO scheme mentioned above. We note that in this step the mass is preserved during the evolution. In the second inner time step of the splitting procedure, we solve the system of ordinary differential equations associated with the reaction contribution, which provides the final approximation of the population densities after one full time step. This variational splitting scheme is described in detail in Section~\ref{sec:reaction}.



The paper is organized as follows: Section \ref{sec:assumption} presents definitions and preliminary facts before laying the assumptions and stating the main result in Theorem \ref{mainth}. In Section \ref{sec:JKO} we consider the constrained system without the reaction part and introduce an energy functional along which we solve the transport equation using the implicit-explicit JKO scheme. In Section \ref{sec:reaction} we demonstrate the splitting scheme to solve the system with the reaction term thus proving the main theorem. The Appendix \ref{sec:appendix} contains a collection of well-known results that are useful in our proofs.

\section{Preliminaries, assumptions and main result}\label{sec:assumption}
We denote by $\mathcal{M}_{+}(\Rnu)$ the set of positive and finite measures and for $p\geq 1$ introduce 
\begin{equation}\label{prob_m}
    \mpdmp{m}:=\left\{\mu\in \mathcal{M}_{+}(\Rnu)\,|\, \mathsf{m}_p[\mu] <+\infty, \, \mu(\Rnu)=m\right\},
\end{equation}
where $\mathsf{m}_p[\mu]$ denotes the  $p-$moment of $\mu$. We endow the space $\mpdmp{m}$ with the $p-$Wasserstein distance type distance,
\begin{equation*}
    \mathrm{d}_{m,p}^p(\mu_1,\mu_2)=\inf_{\pi\in \Pi^m(\mu_1,\mu_2)}\left\{\int_{\Rnu\times\Rnu}|x-y|^p\,d\pi(x,y)\right\},
\end{equation*}
where $\Pi^m(\mu_1,\mu_2)$ is the set of all transport plans $\pi$ between $\mu_1$ and $\mu_2$, that is, the
set of positive measures of fixed mass $\pi(\Rnu\times\Rnu) = m^2$, defined on the product space
with marginals $\mu_1$ and $\mu_2$. When $m=1$ and $p=2$ we will denote $\mpdm{m}=\mpdu$ and
$$
\mathrm{d}_{m,p}^p(\mu_1,\mu_2)=W_2^2(\mu_1\mu_2).
$$
In order to work with densities of evolving mass we recall the definition of Bounded-Lipschitz distance:
\begin{defn}[Bounded-Lipschitz distance]
    Let $\mathcal{M}_+$ denote the set of positive measures. We define the Bounded-Lipschitz distance $d_\mathrm{BL}:\mathcal{M}_+\times\mathcal{M}_+\to \R_+$
    as follows
    \begin{equation}
        d_\mathrm{BL}(\mu,\nu):=\sup \left\{\int_\Rnu f \, \mathrm{d}(\mu-\nu)\quad|\quad\|f\|_{L^\infty(\Rnu)}, \|f'\|_{L^\infty(\Rnu)}\le 1\right\}
    \end{equation}
\end{defn}
We list some fundamental properties of $d_\mathrm{BL}$ below, see \cite{cafasasc} for further details.
\begin{prop}[Properties of $d_\mathrm{BL}$]\label{propdbl}
    Let $\mu,\nu\in L^1_+(\Rnu).$ The following are true:
    \begin{enumerate}[label=(\roman*)]
        \item $d_{\mathrm{BL}}(\mu,\nu)\le \|\mu-\nu\|_{L^1(\Rnu)},$
        \item $ d_{\mathrm{BL}}(\mu,\nu)\le W_1(\mu,\nu):=\inf_{\gamma\in\Pi(\mu_1,\mu_2)}\left\{\int_{\Rnu\times\Rnu} |x-y| d\gamma(x,y)\right\},$ whenever $\mu(\Rnu)=\nu(\Rnu).$
        \item Let $\rho_1,\rho_2,\rho_3\in\mathcal{M}_+(\Omega)\cap L^1(\Omega)$ such that $\|\rho_2\|_{L^1}=\|\rho_3\|_{L^1}.$ Then 
    \begin{equation}
        d_\mathrm{BL}(\rho_1,\rho_3)\le \|\rho_1-\rho_2\|_{L^1(\Omega)}+ W_1(\rho_2,\rho_3).
    \end{equation}
    \end{enumerate}
\end{prop}

In \eqref{maineq} a crucial role is played by the Bessel kernel $K$ that is the kernel of the operator $\Delta-1$.  The kernel $K$ has the following general form
\begin{equation*}
    K(x)=\frac{1}{4\pi} \int_0^\infty t^{-\frac{N}{2}}e^{-\frac{\pi|x|^2}{t}-\frac{1}{4\pi}}\,dt.
\end{equation*}
As a consequence of the expression above it follows that $\|K\|_{L^1(\Rnu)}=\hat{K}(0)=1$, where $\hat{K}$ denotes the Fourier transform of $K$.  
The following lemma lists some fundamental properties of $K$ that we will find useful, refer to \cite{adams1999function,cygan2021stability} for the proofs.
\begin{lem}\cite{cygan2021stability}
    Let $\mathscr{S}'(\Rnu)$ be the space of tempered distributions. Let $\psi\in\mathscr{S}'(\Rnu)$  solve the equation $-\Delta \psi + \psi = f$ for some $f\in \mathscr{S}'(\Rnu).$ Then the following are true:
    \begin{enumerate}[label=(\roman*)]
        \item $\psi = K\ast f ,$ where $K$ is the anti-Fourier transform of  $\hat{K}(\xi) = \frac{1}{1+|\xi|^2}$.
           \item $K$ is radial, depending only on $|x|$ and decreasing with $|x|.$
        \item for $d=1$, $K$ is given by the Morse-type potential $K(x) = \frac{1}{2}e^{-|x|}$.
        \item For $d\ge 2$, $K\in L^1(\Rnu)\cap L^p(\Rnu)$ for each $p\in [1,\frac{d}{d-2})$ and $\nabla K \in L^1(\Rnu)\cap L^q(\Rnu)$ for each $q\in [1, \frac{d}{d-1}).$
    \end{enumerate}
\end{lem}

It is worth noting that in the two-dimensional case the Bessel kernel satisfies the following qualitative properties
\begin{equation}\label{eq:Bess_zero}
   K(x) \sim -\frac{1}{2\pi}\log |x|,\,\mbox{ as }\,|x|\to 0, 
\end{equation}
and
\begin{equation}\label{eq:Bess_inf}
   K(x)\sim \frac{1}{2\sqrt{2\pi}}\frac{e^{-|x|}}{\sqrt{|x|}},\,\mbox{ as }\,|x|\to +\infty, 
\end{equation}
see \cite{adams1999function} for a more detailed introduction to the Bessel kernel.

We are now ready to list the main assumptions on the nonlinear functions involved in \eqref{maineq}.
\begin{ass}\label{ass}
Let $\gamma>1$ and $\Phi \in C^2\bigl([0,+\infty)\bigr)$ be such that $\Phi(s)\geq 0$ for all $s\geq 0$, and that there exist two constants $0<c_\gamma\le C_\gamma$ such that
\begin{equation}\label{growth_phi}
  c_\gamma s^{\gamma-2}\leq  \Phi''(s)\leq C_\gamma s^{\gamma-2}, \qquad \forall s\geq 0.
\end{equation}
For technical reasons, we introduce the function $f:[0,+\infty)\to\R$ defined by
\begin{equation}\label{def:f}
f(s)=\Phi\left(s^{\frac{2}{\gamma}}\right).
\end{equation}
Concerning the reaction term in the first equation of \eqref{maineq} we assume that $M\in C^1([0,\infty))$ and there exists $k_M>0$ such that,
\begin{equation}\label{ass_M}
\frac{M(s)}{s}\leq k_M(1-s)\quad \mbox{for all}\,s>0.   
\end{equation}
We finally assume that $h\in C^1([0,\infty))$, non-negative, $h(s)=0 $ if and only if $s=0$, and there exists $k_h>0$ such that  
\begin{equation}\label{ass_h}
  h'(s)\leq k_h \quad \mbox{for all}\,s\geq0. 
\end{equation}
\end{ass}

We now state the notion of weak solutions that we are dealing with

\begin{defn}\label{def:weak_sol}(Weak solution)
        Let $T>0$ be fixed. A couple $(\rho,\beta)$ with $\rho:[0,T]\to \mpdu$ and $\beta\in L^\infty([0,T]\times\Rnu)$ is a weak solution to \eqref{maineq} if $\rho\in L^\gamma([0,T]\times\Rnu),$ $\nabla\rho^\frac{\gamma}{2}\in L^2([0,T]\times\Rnu)$ and for all $0\leq t<s\leq T$ we have 
    \begin{align*}
    \int_\Rnu \varphi (x)(\rho(s,x)-\rho(t,x))&dx = -\int_t^s \int_\Rnu  \rho(\sigma,x) \nabla\Phi'\left(\rho(\sigma,x)\right)\cdot \nabla\varphi(x) \, dx d\sigma  \\
    &-\frac{\chi}{2}\int_t^s\int_\Rnu \rho(\sigma,x) \int_\Rnu \nabla K(x-y)\cdot \left(\nabla\varphi(x)-\nabla\varphi(y)\right)\rho(\sigma,y) \, dy \, dx \,d\sigma \\
    &-\chi\int_t^s \int_\Rnu \nabla K \ast \beta(\sigma,x)\cdot \nabla\varphi(x)\rho(\sigma,x) \, dx \, d\sigma\\
    &+\int_t^s \int_\Rnu M(\rho(\sigma,x)) \varphi(x) \, dx \, d\sigma,
\end{align*}
and 
\begin{equation*}
    \int_t^s\int_\Rnu h(\rho(\sigma,x))(1-\beta(\sigma,x))\psi(\sigma,x)\,dx\,dt=0.
\end{equation*}
for all $\varphi\in C_c^\infty(\Rnu)$ and  $\psi\in C^\infty([0,T]\times\Rnu)$.
\end{defn}
The main result of the paper reads as follows
\begin{thm}\label{mainth}
 Let $T>0$ be fixed and $\gamma > \max\{1, \tfrac{d-2}{2}\}$ be given and consider $\rho_0 \in  L^1 (\Rnu ) \cap  L^\gamma (\Rnu )$. Assume that the nonlinear functions $\Phi$, $M$ and $h$  satisfy
the properties in Assumptions \ref{ass}. Then,
there exists a couple $(\rho,\beta)$ that is a weak solution of system \eqref{maineq} in the sense of Definition \ref{def:weak_sol}.
\end{thm}

\section{Equation without reaction: the JKO scheme}\label{sec:JKO}
In this section, we consider the reaction-free version of \eqref{maineq}: 
\begin{equation}\label{eq:noreaction}
    \begin{cases}
     \partial_t\rho = \nabla\cdot \left( \rho\nabla\Phi'(\rho) -\chi(\rho\nabla K\ast (\rho + \beta)\right) \quad \text{in}\quad x\in\Rn, t>0,\\
     h(\rho)\left(1-\beta\right)=0,\quad x\in\Rn.
\end{cases}
\end{equation}

Driven by the constraint, we define an \emph{energy} functional for the PDE in \eqref{eq:noreaction}. We construct approximate solutions utilizing the minimizing movement scheme introduced by Jordan-Kinderlehrer-Otto. We obtain a minimizer for a given time-step and establish \textit{a priori} bounds that are independent of the time-step, which facilitates the passage to the limit to obtain a weak solution. The notion of weak solutions we want to catch for the system in \eqref{eq:noreaction} is the following:
\begin{defn}\label{weak_jko}
    Let $T>0$ be fixed. A couple $(\rho,\beta)$ with $\rho:[0,T]\to \mpdu$ and $\beta\in L^\infty([0,T]\times\Rnu)$ is a weak solution to \eqref{eq:noreaction} if $\rho\in L^\gamma([0,T]\times\Rnu),$ $\nabla\rho^\frac{\gamma}{2}\in L^2([0,T]\times\Rnu)$ and for all $0\leq t<s\leq T$ we have 
    \begin{align*}
    \int_\Rnu \varphi (x)(\rho(s,x)-\rho(t,x))dx& = -\int_t^s \int_\Rnu  \rho(\sigma,x)\nabla\Phi'\left(\rho(\sigma,x)\right)\nabla\varphi(x) dx \,d\sigma  \\
    &-\frac{\chi}{2}\int_t^s\int_\Rnu \rho(\sigma,x) \int_\Rnu \nabla K(x-y)\cdot (\nabla\varphi(x)-\nabla\varphi(y))\rho(\sigma,y) \, dy \, dx \, d\sigma \\
    &-\chi\int_t^s \int_\Rnu \nabla K \ast \beta(\sigma,x)\cdot \nabla\varphi(x)\rho(\sigma,x) \, dx \, d\sigma,
\end{align*}
and 
\begin{equation*}
    \int_t^s\int_\Rnu h(\rho(\sigma,x))(1-\beta(\sigma,x))\psi(\sigma,x)\,dx\,dt=0.
\end{equation*}
for all $\varphi\in C_c^\infty(\Rnu)$ and  $\psi\in C^\infty([0,T]\times\Rnu)$.
\end{defn}
Given two probability measures $\rho,\nu \in \mpdu, $ we introduce the following functional:
\begin{equation}\label{eq:im_ex_fun}
    \mF[\rho|\nu] =\mathscr{A}[\rho]+\mathscr{K}[\rho]+\mathscr{S}[\rho|\nu]\coloneq \int_\Rnu \Phi(\rho) \,dx - \frac{\chi}{2}\int_\Rnu\rho K\ast \rho\,dx - \chi \int_\Rnu \rho K\ast \mathbb{1}_{S_\nu}\,dx,
\end{equation}
where 
\begin{equation}
    S_\nu=\left\{x\in\Rnu \,:\,\nu(x)>0\right\},
\end{equation}
and $\mathbb{1}_{S_\nu}$ denotes the characteristic function of the set, namely
\begin{equation*}
    \mathbb{1}_{S_\nu}(x)=\begin{cases}
        1, & \mbox{if } x\in S_\nu,\\
        0, & \mbox{otherwise}.
    \end{cases}
\end{equation*}
Consider an initial density $\rho^0\in\mpdu$ with $\mF[\rho^0|\rho^0]< +\infty$. Fix $T>0$ and let $\tau>0$ and $N \in \mathbb{N}$ be such that $\tau N=T.$ Let $\mathcal{P}_\tau:=\{k\tau|k\in\N, \, 0\le k\le N\}$ be a uniform partition of the time-domain $[0,T].$ We construct the following recursive sequence of measures in $\mpdu$: we set $\rho_\tau^0=\rho^0$ and for any $k\geq 1$
\begin{equation}\label{JKO}
\rho_\tau^k \coloneq\argmin_{\rho\in\mpdu}\left\{\frac{1}{2\tau}W_2^2(\rho_\tau^{k-1},\rho)+\mF[\rho|\rho_\tau^{k-1}]\right\}.
\end{equation}
\begin{rem}
    The well posedness of the above scheme can be obtained from classical arguments in Calculus of Variations. We start observing that by Young inequality for convolutions
    \begin{align*}
   \int_\Rnu \rho(x) K\ast \mathbb{1}_{S_\nu}(x) dx & \le \left|\int_\Rnu \rho(x)  K\ast \mathbb{1}_{S_\nu}(x)(x)dx\right|\\
    & \leq \|\rho\|_{L^1(\Rnu)}\|K\ast \mathbb{1}_{S_\nu}\|_{L^\infty(\Rnu)}\\
     & \leq \|\rho\|_{L^1(\Rnu)}\|K\|_{L^1(\Rnu)}=1.
\end{align*}
Thus, $\mathscr{S}[\rho|\nu]$ is always bounded from below, independently of $\nu$.
For the interaction potential $\mathscr{K}[\rho]$, given 
\begin{equation*}
    \gamma > \max\left\{1, \tfrac{d-2}{2}\right\},
\end{equation*}  
and assuming that $\mathscr{A}[\rho]< +\infty$, we use H\"{o}lder inequality and Young inequality for convolutions to estimate
\[\|\rho K\ast \rho\|_{L^1(\Rnu)} \le \|K\|_{L^r(\Rnu)}\|\rho\|_{L^q(\Rnu)}.\]
for some $r\in \left[1,\tfrac{d}{d-2}\right)$ and $q=\tfrac{r}{r-1}\in [1,\gamma]$

The existence of an optimizer in \eqref{JKO} can be deduced from similar results in \cite{BlCaCa,CaSa}. 
Weak $L^1$ convergence for such a sequence follows from Carleman type inequality, Dunford-Pettis theorem and the uniform control on the second moments of $\rho_\tau^k$, see \cite{CaSa}[Section 2.1], \cite{BlCaCa}[Lemma 2.1 and Lemma 3.1].
\end{rem}

For the sequence $\left\{\rho_\tau^k\right\}_k$ constructed in \eqref{JKO} we introduce the piecewise constant interpolation  $$\rho_\tau:[0,T]\to \mpdu$$ as follows:
\begin{equation}\label{eq:piec_den_rho}
\rho_\tau(t)=\rho_\tau^k, \quad\text{whenever} \quad t\in\left((k-1)\tau,k\tau\right],1\le k \le N.
\end{equation}
\begin{prop}\label{narrow_conv}
    Let $T>0$ be fixed. There exists an absolutely continuous curve $\rho :[0,T]\to \mpdu$ such that the piecewise constant interpolation $\rho_\tau$ introduced in \eqref{eq:piec_den_rho} narrowly converges (up to a subsequence) to $\rho$ uniformly in $[0,T]$ as $\tau \to 0$.
\end{prop}
\begin{proof}
    From the definition of a minimizer in \eqref{JKO}, we have for any $k\in\N, 1\le k\le N,$
\begin{equation} \label{temp0}
\frac{1}{2\tau} W_2^2(\rho_\tau^{k-1},\rho_\tau^k)\le  \mF\left[\rho_\tau^{k-1}|\rho_\tau^{k-1}\right]-\mF\left[\rho_\tau^k|\rho_\tau^{k-1}\right].
\end{equation}
Let $\pi_\tau^k$ be the optimal transport plan between $\rho_\tau^k$ and $\rho_\tau^{k-1}$ corresponding to the quadratic cost function. We note that
\begin{align*}
  \mathscr{S}\left[\rho_\tau^{k-1}|\rho_\tau^{k-1}\right]-\mathscr{S}\left[\rho_\tau^k|\rho_\tau^{k-1}\right]&\le \chi\int_\Rnu |\rho_\tau^k(x)-\rho_\tau^{k-1}(x)||K\ast \mathbb{1}_{S^{k-1}_\tau}(x)|\,dx\\
   &\le\chi\|\nabla K\|_{L^1(\Rnu)} \iint_{\Rnu\times \Rnu} |x-y|\,d\pi_{\tau}^k(x,y)\\
   & \le \chi W_1(\rho_\tau^k,\rho_\tau^{k-1})
   \le \chi W_2(\rho_\tau^k,\rho_\tau^{k-1})\le \frac{1}{4\tau}W_2^2(\rho_\tau^k,\rho_\tau^{k-1})+\chi^2\tau.
\end{align*}

For any $m,n\in\N$ with $1\le m<n \le N$ summing \eqref{temp0} over all $k$ between $m+1$ and $n$ we get
\begin{equation*}
   \frac{1}{4\tau}\sum_{k=m+1}^n W_2^2(\rho_\tau^k,\rho_\tau^{k-1})\le \mathscr{A}(\rho_\tau^m)-\mathscr{A}(\rho_\tau^n)+\mathscr{K}(\rho_\tau^m)-\mathscr{K}(\rho_\tau^n)+\chi^2(n-m)\tau.
\end{equation*}
Since the functional $\mathscr{A}$ and $\mathscr{K}$ are bounded from below we obtain the following estimate for any $n\in \mathbb{N},\, n\le N$
\begin{equation}\label{estim_W_2}
\sum_{k=1}^n W_2^2(\rho_\tau^k,\rho_\tau^{k-1})\le 4\tau\left(\mathscr{A}(\rho^0)+\mathscr{K}(\rho^0)\right)+C\tau+ 4n\chi^2\tau^2.
\end{equation}
Regarding the second moments, we observe that
\begin{align*}
\mathrm{m}_2[\rho_\tau^k]:=\int_{\Rnu} |x|^2 d\rho_\tau^k
    &=\iint_{\Rnu\times \Rnu} |x|^2 d\pi_\tau^k(x,y)\\
    &\le 2 \iint_{\Rnu\times \Rnu} |x-y|^2d\pi_\tau^k(x,y) + 2 \iint_{\Rnu\times \Rnu} |y|^2d\pi_\tau^k(x,y)\\
    &\le 2 W_2^2(\rho_\tau^k,\rho_\tau^{k-1}) + 2 \int_\Rnu |y|^2 d\rho_\tau^{k-1}(y)\\
    &\le 2 \sum_{i=1}^N W_2^2(\rho_\tau^k,\rho_\tau^{k-1}) + 2\mathrm{m}_2[\rho^{k-1}_\tau]\\
    & \le 8\tau\left(\mathscr{A}(\rho^0)+\mathscr{K}(\rho^0)\right)+ 2C\tau + 8 N \chi^2 \tau^2 + 2^N \mathrm{m}_2[\rho^0].
\end{align*}
The second moments $\rho_\tau$ on compact time-intervals are thus uniformly bounded. This implies pre-compactness with respect to the narrow topology on $\mpdu.$ 
Now, for any $0 < s< t < T,$ say $s\in[(m-1)\tau,m\tau)$  and $t\in[n\tau,(n+1)\tau)$ for some $m,n \in \mathbb{N},$ we have
\begin{align*}
W_2(\rho_{\tau}(s),\rho_{\tau}(t)) 
    & \le \sum_{j=m+1}^n W_2(\rho_{\tau}^{j-1},\rho_{\tau}^j)\le \left(\sum_{j=m+1}^n W_2^2(\rho_{\tau}^{j-1},\rho_{\tau}^j)\right)^{\frac{1}{2}}(n-m)^{\frac{1}{2}}\\
    &\le \left(4\tau\left(\mathscr{A}(\rho^0)+\mathscr{K}(\rho^0)\right)+C\tau+4 N\chi^2\tau^2\right)^{\frac{1}{2}}\left(n-m)\right)^{\frac{1}{2}}\label{boundintau}\\
    &\le C\left(t-s\right)^\frac{1}{2}+o(\tau).
\end{align*}
Using the refined version of Ascoli-Arzel\`{a} theorem in Proposition \ref{ascoli}, for any given sequence $\tau_k\to 0,$ we obtain a (non-relabelled) subsequence such that  $\left\{\rho_{\tau_{k}}\right\}_k$ converges narrowly to a H\"{o}lder continuous limit curve $\rho:[0,T]\to \mpdu$ as $k\to +\infty.$
\end{proof}

Narrow convergence is insufficient to pass into the limit in the non-linear terms. We require the sequence to converge strongly in some $L^p$ space. We utilize the so called flow interchange technique \cite{matthes2009family} to obtain suitable bounds on the spatial derivatives of the discrete approximations. We introduce an auxiliary functional 
\[\mathscr{H}[\eta] = \int_\Rnu \left(\eta(x) \log \eta(x) - \eta(x)\right)\,dx\] 

It is well known that this \textit{entropy} functional possesses a flow $G_\mathscr{H}$ given by the heat semi-group. More precisely, for a given $\eta_0\in\mpdu,$ the curve $s\mapsto \eta(s)\coloneq G_\mathscr{H}^s\eta_0$ solves the initial value problem 
$$\partial_t\eta=\Delta \eta,\quad \eta(0)=\eta_0,$$
in the classical sense: $\eta(s)$ is a density function, for every $s>0,$ $\eta\in C^1\left((0,\infty);C^\infty(\Rnu)\cap L^1(\Rnu)\right),$ and if $\eta_0\in L^p(\Rnu),$ then $\eta(s)$ converges to $\eta_0$ in $L^p(\Rnu)$ as $s\downarrow 0.$

\begin{lem}\label{regularity}
    There exists a constant $C$ depending only on $\rho^0$ such that the piecewise constant interpolations $\rho_\tau$ satisfy
    \begin{equation}
        \left\|\rho_\tau^{\frac{\gamma}{2}}\right\|_{L^2(0,T;H^1(\Rnu))}\le C(1+T).
    \end{equation}
\end{lem}
\begin{proof} We start by evaluating the dissipation of the functional $\mF$ along the flow associated with the functional $\mathscr{H}$. By shortening the notation and calling $\eta \coloneq G_\mathscr{H}^t\rho^0$ we have
\begin{align*}
    \frac{\dd}{\dt}\mF[\eta|\rho_\tau^{k-1}] &= \int_\Rnu \left( \Phi'(\eta) - \chi K\ast\eta - \chi K\ast \mathbb{1}_{S^{k-1}_\tau}\right)\Delta\eta\,dx\\
    & = -\int_\Rnu \Phi''(\eta)\nabla \eta\cdot \nabla\eta \,dx+ \chi\int_\Rnu \nabla\eta\cdot \nabla K \ast \eta \,dx+ \chi\int_\Rnu \nabla\eta\cdot\nabla K \ast \mathbb{1}_{S^{k-1}_\tau}\,dx\\
    &\leq -C_\gamma\int_\Rnu \eta^{\gamma-2}|\nabla\eta|^2 \,dx+ \chi\int_\Rnu \nabla\eta\cdot \nabla K \ast \eta\,dx + \chi\int_\Rnu \nabla\eta\cdot\nabla K \ast \mathbb{1}_{S^{k-1}_\tau}\,dx\\
    &= -\frac{4C_\gamma}{\gamma^2} \int_\Rnu \left|\nabla\eta^{\frac{\gamma}{2}}\right|^2 \,dx+ \chi\int_\Rnu \nabla\eta\cdot \nabla K \ast \eta \,dx+ \chi\int_\Rnu \nabla\eta\cdot\nabla K \ast \mathbb{1}_{S^{k-1}_\tau}\,dx
\end{align*}
where we used the bound in \eqref{growth_phi}. Using the Cauchy Schwarz inequality we estimate,
\begin{align*}
    \int_\Rnu \nabla\eta \cdot \nabla K \ast \eta\, dx &= \int_\Rnu \nabla \eta \cdot K\ast \nabla \eta \, dx \\
    & \le \|K\ast\nabla \eta\|_{L^2(\Rnu)} \|\nabla\eta\|_{L^2(\Rnu)}\\
    &\le \|K\|_{L^1(\Rnu)}\|\nabla\eta\|_{L^2(\Rnu)}^2=\|\nabla\eta\|_{L^2(\Rnu)}^2.
\end{align*}
Using the properties of the Bessel kernel we have,
\begin{align*}
    \int_\Rnu \nabla\eta\cdot \nabla K \ast \mathbb{1}_{S^{k-1}_\tau} \, dx &= \int_\Rnu \eta \left(\delta_0-K\right)*\mathbb{1}_{S^{k-1}_\tau} \, dx\\
    &=\int_\Rnu \eta \, \mathbb{1}_{S^{k-1}_\tau} \, dx - \int_\Rnu \eta \, K*\mathbb{1}_{S^{k-1}_\tau} \, dx\\
    &\leq  1+\|\eta\|_{L^1}\|K*\mathbb{1}_{S^{k-1}_\tau}\|_{L^\infty}\leq 2.
\end{align*}
Since $\eta\in L^\infty(0,T;H^1_0(\Rnu)),$ $\|\nabla\eta\|_{L^2(\Rnu)}^2$ is bounded, therefore
\begin{align*}
     \frac{\dd}{\dt}\mF[\eta|\rho_\tau^{k-1}] &\le -\frac{4C_\gamma}{\gamma^2} \int \left|\nabla\eta^{\frac{\gamma}{2}}\right|^2 + C(T,\rho^0).
\end{align*}
Define 
\begin{equation*}
    \mathfrak{R}[\eta]:=\frac{4C_\gamma}{\gamma^2}\int \left|\nabla\eta^{\frac{\gamma}{2}}\right|^2 - C(T,\rho^0).
\end{equation*}
We note that
\begin{equation*}\liminf_{s\downarrow 0}\left(-\frac{\dd}{\dt}\mF[G_\mathscr{H}^t\rho^0|\rho_\tau^{k-1}]\big|_{t=s}\right)\ge \mathfrak{R}[\eta].
\end{equation*}
Then from Corollary \ref{flowintercor} it follows that
\begin{equation}\label{temp1}
    \mathscr{H}[\rho_\tau^N]\le \mathscr{H}[\rho^0] - \tau \sum_{k=1}^N \mathfrak{R}[\rho_\tau^k]
\end{equation}
Using standard inequalities for any $s>0,$ 
\begin{equation*}
-\frac{2}{e}\sqrt{s}\le s\log s\le \frac{1}{(\gamma-1)e}s^\gamma,
\end{equation*}
we have that
\begin{align*}
    \mathscr{H}[\eta] &= \int_\Rnu (\eta\log\eta-\eta) \le \frac{1}{(\gamma-1)e}\int_\Rnu \eta^\gamma \le C_1\left(\mF[\eta|\rho_\tau^{k-1}]+C_2\right),
\end{align*}
as well as,
\begin{align*}
    \mathscr{H}[\eta]&\ge -\frac{2\sqrt{\pi}}{e}\left(1+\int_\Rnu |x|^2\eta\right)^\frac{1}{2}.
\end{align*}
From \eqref{temp1} we then obtain
\begin{equation*}
    \tau \sum_{k=1}^N \mathscr{R}[\rho_\tau^k] \le C_1\left(\mF[\rho^0]+C_2\right)+C_3\left(1+\int_\Rnu |x|^2\rho_\tau^N \, dx \right)^\frac{1}{2} \le C.
\end{equation*}
Thus, we have a bound on the $H^1-$ norm of $\rho^{\frac{\gamma}{2}}$
\begin{align*}
    \frac{4C_\gamma}{\gamma^2}\tau \sum_{k=1}^N \left\|(\rho_\tau^k)^{\frac{\gamma}{2}}\right\|_{H^1}^2,
    &\le C(1+T)
\end{align*}
where the constant $C$ depends only on $\rho^0$ and $\|\nabla K\|_{L^1}.$
Thus, it follows  that 
\begin{equation}
    \int_0^T \left\|\rho_\tau(t)^{\frac{\gamma}{2}}\right\|_{H^1(\Rn)}^2 \,dt = \tau \sum_{k=1}^N \left\|(\rho_\tau^k)^\frac{\gamma}{2}\right\|_{H^1(\Rn)}^2 \le C(1+T).
\end{equation}
\end{proof}

Now, we proceed to show that the narrowly converging subsequence of $\left\{\rho_\tau\right\}_{\tau\ge 0}$ converges to a limit $\rho$ in $L^{\gamma}([0,T]\times \Rnu)$. We exploit the extended Aubin-Lions Lemma \ref{aubin}.

\begin{prop}\label{prop:conv}
    The converging sub-sequence $\rho_{\tau_k}$ from Proposition \ref{narrow_conv} converges to a limit function $\rho$ in $L^\gamma([0,T]\times\Rnu)$.
\end{prop}
\begin{proof}
Let $U=\left\{\rho_{\tau_k}\right\}_k$. From the estimates in Lemma \ref{regularity} we note that for any $\rho\in U,$ $\rho^{\frac{\gamma}{2}}\in L^2\left([0,T];H^1(\Rnu)\right).$ We define a functional $\mathscr{Y}:L^\gamma(\Rnu)\to [0,\infty],$ 
\begin{equation*}
    \mathscr{Y}[\rho] = \left\{ 
    \begin{matrix}
        \|\nabla\rho^{\frac{\gamma}{2}}\|_{L^2(\Rnu)}^2 + \mathrm{m}_2[\rho], & \quad\text{if both}\quad \nabla\rho^{\frac{\gamma}{2}}\in L^2(\Rnu), \mathrm{m}_2[\rho]<+\infty, \\
        +\infty, & \text{otherwise,}
    \end{matrix}
    \right.
\end{equation*}
and a distance function $\mathrm{d}:L^\gamma(\Rnu)\times L^\gamma(\Rnu)\to \mathbb{R}\cup \{+\infty\},$ 
\begin{equation*}
    \mathrm{d}(\eta,\rho) = \left\{\begin{matrix}
        W_2(\eta,\rho) & \text{~~~~if~~~} \eta,\rho\in\mathscr{P}_2(\Rnu)\\
        +\infty, & \text{otherwise}
    \end{matrix}
    \right.
\end{equation*}
Clearly, $\mathscr{Y}$ is lower semi-continuous and $\mathrm{d}$ is a pseudo-distance.
We claim that the sub-level sets $\mathscr{Y}_C=\{\rho\in L^\gamma(\Rnu) | \mathscr{Y}[\rho]\le C\},$ for any $C>0$ are pre-compact in $L^\gamma(\R^n).$ It is sufficient to show that the associated sets $\mathscr{Z}_C:=\{\eta=\rho^\frac{\gamma}{2}|\rho\in\mathscr{Y}_C\}$ are pre-compact in $L^2(\Rnu).$  Since $\|\nabla\eta\|_{L^2(\Rnu)}\le C$ it is easy to see that $\mathscr{Z}_C$ is $L^2-$equicontinuous.
For the $L^2-$equitightness, given $R>0,$ we take $0 < \delta < 1,$ and estimate
\begin{align*}
    \int_{|x|>R}\eta(x)^2\, dx &\le \frac{1}{R^{2\delta}} \int_\Rnu |x|^{2\delta} \eta(x)^2\, dx\\
    & = \frac{1}{R^{2\delta}}\int_\Rnu \left(|x|^2\eta(x)^\frac{2}{\gamma}\eta(x)^{\frac{2}{\delta}-\frac{2}{\gamma}}\right)^\delta \,dx\\
    &\le \frac{1}{R^{2\delta}} \left(\int_\Rnu |x|^2\eta^\frac{2}{\gamma}(x)\,dx\right)^\delta \left(\int_\Rnu \eta(x)^{\frac{2(\gamma-\delta)}{\gamma(1-\delta)}}\,dx\right)^{1-\delta}\\
    &\le \frac{1}{R^{2\delta}} C^\delta C_{GNS} \|\nabla \eta\|_{L^2(\Rnu)}^{\frac{2 \delta(\gamma-1)}{\gamma}} \|\eta\|_{L^2(\Rnu)}^{2(1-\delta)}
\end{align*}
where we made careful use of Gagliardo-Nirenberg interpolation inequality and used the uniform control on the second moment of $\rho$.
Then by Fr\'{e}chet-Kolmogorov Theorem,    $\mathscr{Z}_C$ is pre-compact in $L^2(\Rnu)$. 

From Lemma \ref{regularity} we have a uniform bound $$\sup_{\rho\in U}\int_0^T \mathscr{Y}[\rho(t)]dt \le T\left(C(1+T) + C(T,\rho^0)\right) <\infty. $$
Further, from the H\"{o}lder continuity of the family $U$ we have for any $\rho\in U,$
\begin{align*}
\limsup_{h\downarrow 0} \int_0^ {T-h} \mathrm{d}(\rho(t+h),\rho(t)) \,dt &= \limsup_{h\downarrow 0} \int_0^ {T-h} W_2(\rho(t+h),\rho(t)) \,dt \\
&\le C \limsup_{h\downarrow 0} \int_0^{T-h}\sqrt{h} \,dt =0.
\end{align*}
Then by Aubin-Lions Lemma \ref{aubin}, there exists a sub-sequence $\{\rho_{\tau_{h}}\}_{h}$ converging in measure (with respect to $t\in [0,T])$ to some limit $\rho^\ast:[0,T]\to L^\gamma(\R^2).$
But this limit coincides with the narrow limit $\rho$ from Proposition \ref{narrow_conv}. Indeed, the entire sequence $\{\rho_{\tau_k}\}_k$ converges to $\rho$ in measure.
Since we have a $\tau-$independent uniform $L^\gamma$ estimate on $\rho_\tau$ by virtue of 
$$\|\rho_\tau^k\|_{L^\gamma}\le \mF[\rho_\tau^k| \rho_\tau^{k-1}]\le \mF[\rho^0|\rho^0],$$
we use the Dominated Convergence theorem to get the strong convergence of the sequence $\{\rho_{\tau_k}\}_k$ in $L^\gamma\left(0,T;L^\gamma(\Rnu)\right),$
$$\int_0^T \|\rho_{\tau_k}(t,\cdot)-\rho(t,\cdot)\|_{L^\gamma(\Rnu)}^\gamma \, dt \to 0$$ as $k\to \infty.$
\end{proof} 

\begin{rem}
Equitightness and $L^\gamma-$convergence in Proposition \ref{prop:conv} easily implies $L^1$ strong convergence of $\rho_\tau$ to $\rho$. 
\end{rem}

We now introduce a piecewise constant interpolation function $\beta_\tau:[0,\infty)\to L^\infty(\Rnu)$ as follows:
\begin{equation}\label{eq:piec_den}\beta_\tau(t,x)=\mathbb{1}_{S^{k}_\tau}(x), \quad\text{whenever} \quad t\in\left((k-1)\tau,k\tau\right),k\ge 1.
\end{equation}
\begin{prop}\label{d_conv}
The piecewise constant interpolating function $\beta_\tau$ defined in \eqref{eq:piec_den} converges  weakly-$\ast$ in $L^\infty([0,T]\times\R^2)$ to a certain function $\beta$. Moreover, the limit function $\beta$ satisfies
\begin{equation*}
    h(\rho)(1-\beta)=0,\quad  \mbox{a.e.}\,(t,x),
\end{equation*}
with $\beta=1,$ $ \rho-\mbox{a.e.}\,(t,x)$ and $\beta(t,x)=0 \ \mbox{a.e.} \text{ in } \{(t,x): \rho(t,x)=0\}$.
\end{prop}
\begin{proof}
Since $h(0)=0$, the couple $(   \rho_\tau^k,\mathbb{1}_{S^{k}_\tau})$ obviously satisfies 
\begin{equation}\label{eq:constr_discr}
    h(\rho_\tau^k(x))(1-\mathbb{1}_{S^{k}_\tau}(x))=0,\quad\mbox{ for all }x\in \Rnu,
\end{equation}
for all $k\geq1$. By construction
\begin{equation*}
\beta_\tau \overset{\ast}{\rightharpoonup} \beta
\end{equation*}
as $\tau\to0$ in $L^\infty\left(\left[0,T\right]\times\Rnu\right)$ to some bounded function $\beta$. Consider a test function $\psi\in C^\infty(\left[0,T\right]\times\Rnu)$, then from \eqref{eq:constr_discr} we have
\begin{equation}\label{eq:perweak}
    \int_0^T\int_\Rnu h(\rho_\tau)(1-\beta_\tau)\psi(t,x)\,dx\,dt=0.
\end{equation}
Then, the convergence to $\beta$, results in Proposition \ref{prop:conv} and assumptions on $h$ allow to pass to the limit in the previous equality in order to deduce 
\begin{equation}\label{eq:constrcont}
 h(\rho)(1-\beta)=0,\quad  \mbox{a.e.}-(t,x),
\end{equation}
and
\begin{equation}\label{eq:deqone}
 \beta=1,\quad  \rho-\mbox{a.e.}-(t,x).
\end{equation}
In order to conclude, we need to prove that 
\begin{equation*}
\beta=0, \quad \mbox{ a.e. in }\mathcal{S}_\rho \coloneq \left\{(t,x)\in [0,T]\times\Rnu \,: \,\rho(t,x)=0\right\}.
\end{equation*}
If not, assume there exist an open set $B\subset \mathcal{S}_\rho$ such that $0<\beta(t,x)\leq 1$  a.e.$-(t,x)$ in $B$. 
Thanks to \eqref{eq:piec_den} and \eqref{eq:constr_discr} we have that
\[
\left\{(t,x)\in B \,: \, \beta_\tau(t,x)> 0 \right\} = \left\{(t,x)\in B \,: \, \beta_\tau(t,x) =1 \right\} = B \, \cap S_{\rho_\tau} = \left\{(t,x) \in B \, :\, \rho_\tau(t,x) > 0 \right\}.
\]
Since the $L^1$ convergence of  $\rho_\tau$ to $\rho$ implies convergence in measure we have that for all $\delta >0 $
\[
 \mu\left(\left\{(t,x)\in B\,:\,\rho_\tau(t,x)>\delta\right\}\right)\to 0,\,\mbox{ as }\,\tau\to 0.
\]
This gives 
\[
\mu\left(\left\{(t,x)\in B \,: \, \beta_\tau(t,x)> 0 \right\}\right) \to 0 \mbox{ as }\,\tau\to 0,
\]
that is, $\beta_\tau\to 0$ in measure in $B$. This together with the $L^\infty$ weak$-\ast$ convergence implies $\beta=0$ a.e. in $B$. Indeed, convergence in measure ensures that there exists a subsequence $\beta_{\tau_l}\to 0$ a.e. in $B$. Let $g\in L^1(B)$, then dominated convergence gives
\[
\int_B \beta_{\tau_l}(t,x)g(t,x)\,dx\,dt\to 0.
\]
Since weak$-\ast$ limit should be the same along subsequences
\[
\int_B \beta(t,x)g(t,x)\,dx\,dt= 0,
\]
which is a contradiction to our assumption $\beta>0$ a.e. in $B$.
\end{proof}

\begin{prop}\label{cons_JKO}
    The interpolating sequence $\left\{\left(\rho_\tau,\beta_\tau\right)\right\}$ converges, up to a subsequence, to a weak solution $(\rho,\beta)$ of (\ref{maineq}) in the sense of Definition \ref{weak_jko}.
\end{prop}
\begin{proof} Let $\rho_0, \rho$ be two consecutive minimizers of \eqref{JKO} corresponding to a time-step $\tau>0.$ Let $T$ be the optimal transport map between $\rho_0$ and $\rho$, that is, $\rho=T_\#\rho_0.$ Given $\epsilon>0$ and a test function $\zeta\in C_c^\infty(\Rnu),$ define a map $P^\epsilon(x)=x+\epsilon\zeta(x)$ and let $\rho^\epsilon\coloneq P^\epsilon_\#\rho.$ Since $\rho$ is a minimizer we have 
\begin{equation}\label{1}
    \frac{1}{2\tau}\left(W_2^2(\rho^\epsilon,\rho_0)-W_2^2(\rho,\rho_0)\right)+\mF[\rho^\epsilon|\rho_0]-\mF[\rho|\rho_0] \ge 0.
\end{equation}
We note that
\begin{align*}
    W_2^2(\rho^\epsilon,\rho_0) &\le \int_\Rnu |x-P^\epsilon(T(x))|^2\rho_0(x)\,dx,\\
    W_2^2(\rho,\rho_0) &= \int_\Rnu |x-T(x)|^2\rho_0(x)\,dx,
\end{align*}
and so standard computations gives the following:
\begin{align*}
    \nonumber\frac{1}{2\tau}\left(W_2^2(\rho^\epsilon,\rho_0)-W_2^2(\rho,\rho_0)\right)&\le \frac{1}{2\tau}\int_\Rnu \left(|x-P^\epsilon(T(x))|^2 - |x-T(x)|^2\right)\rho_0(x)\,dx\\
    \nonumber &= \frac{1}{2\tau}\int_\Rnu \left(|x-T(x)-\epsilon\zeta(T(x))|^2-|x-T(x)|^2\right)\rho_0(x)\,dx
\end{align*}
that results in
\begin{align}\label{2}
    \frac{1}{2\tau}\left(W_2^2(\rho^\epsilon,\rho_0)-W_2^2(\rho,\rho_0)\right) & \le -\frac{\epsilon}{\tau}\int_\Rnu (x-T(x))\cdot \zeta(T(x))\rho_0(x)\, dx + o(\epsilon). 
\end{align}
Concerning the terms involving the functional, we have
\begin{align}\label{3}
\mF[\rho^\epsilon| \rho_0]-\mF[\rho| \rho_0] &= \mathscr{A}[\rho^\epsilon]-\mathscr{A}[\rho]+\mathscr{K}[\rho^\epsilon]-\mathscr{K}[\rho]+\mathscr{S}[\rho^\epsilon| \rho_0]-\mathscr{S}[\rho| \rho_0],
\end{align}
and we analyze them term by term. A second order Taylor expansion of $\Phi$ around $\rho$ gives us
\begin{align*}
    \mathscr{A}[\rho^\epsilon]-\mathscr{A}[\rho] =& \int_\Rnu \Phi\left(\rho^\epsilon(x)\right)\, dx - \int_\Rnu \Phi(\rho(x))\,dx\\
    =& \int_\Rnu \Phi\left(\rho^\epsilon(x)\right)\, dx - \int_\Rnu \Phi(\rho(x))\, dx\\
    =& \int_\Rnu \Phi\left(\frac{\rho(x)}{|\det(\nabla P^\epsilon(x))|}\right) |\det(\nabla P^\epsilon(x))|\, dx - \int_\Rnu \Phi(\rho(x))\,dx\\
    =& \int_\Rnu \Phi\left(\rho(x)\right) \left( |\det(\nabla P^\epsilon(x))|-1 \right) \, dx\\
     &+\int_\Rnu \Phi'\left(\rho(x)\right)\rho(x)\left(1-|\det(\nabla P^\epsilon(x))|\right)\, dx +R \\
    \le & \int_\Rnu \left(\Phi\left(\rho(x)\right) - \rho(x) \Phi'\left(\rho(x)\right)\right)\left(|\epsilon \nabla \cdot \zeta(x)| +o(\epsilon) \right) \, dx + R
\end{align*} 
where we expanded the determinant as $\det(\nabla P^\epsilon(x))=1+\epsilon\nabla\cdot\zeta(x)+o(\epsilon)$. The remainder term $R$ is 
\begin{align*}
    R& = \frac{1}{2} \int_\Rnu \Phi''(\tilde \rho^\epsilon(x))\rho(x)^2 \frac{\left(|\det(\nabla P
    ^\epsilon(x))|-1\right)^2}{|\det(\nabla P^\epsilon(x))|} dx.
\end{align*}
We note that
\[
    \frac{\left(|\det(\nabla P
    ^\epsilon(x))|-1\right)^2}{|\det(\nabla P^\epsilon(x))|} \le \frac{\epsilon^2|\nabla\cdot \zeta(x)|^2 +o(\epsilon^2)}{|1+\epsilon \nabla\cdot \zeta(x)+o(\epsilon)|} = o(\epsilon).
\]
This results in $R$ vanishing as $\epsilon\to 0^+.$ 
The nonlocal term can be treated in the standard way
\begin{align*}
 \mathscr{K}[\rho^\epsilon]-\mathscr{K}[\rho] &  = -\frac{\chi}{2}\iint_{\Rnu\times\Rnu} \left(K(P^\epsilon(x)-P^\epsilon(y))-K(x-y)\right)\rho(x)\rho(y)\,dy \,dx \\
    & = -\frac{\chi}{2}\iint_{\Rnu\times\Rnu} \Big( K\big(x-y+\epsilon(\zeta(x)-\zeta(y))\big)-K(x-y)\Big) \rho(y)\rho(x) \,dy\,dx.
\end{align*}
Dividing by $\epsilon$ and sending $\epsilon\to 0^+$ we have  
\begin{align*}
    \lim_{\epsilon\to 0}&\iint_{\Rnu\times\Rnu} \left( \frac{K\big(x-y+\epsilon\left(\zeta(x)-\zeta(y)\right)\big)-K(x-y)}{\epsilon}\right) \rho(y)\rho(x) \,dy\,dx\\
    &=\iint_{\Rnu\times\Rnu}  \nabla K\left(x-y\right)\cdot\left(\zeta(x)-\zeta(y)\right)\rho(y)\rho(x) \,dy\,dx.
\end{align*}
Concernig the last term we first compute
\begin{align*}
   \mathscr{S}[\rho^\epsilon| \rho_0]-\mathscr{S}[\rho| \rho_0] &= -\chi\int_\Rnu \big(K\ast\1(P^\epsilon(x))-K\ast\1(x)\big)\rho(x) \, dx\\
    & =-\chi\int_\Rnu \big(K\ast\1(x+\epsilon\nabla\zeta(x))- K\ast\1(x)\big)\rho(x)\,dx.
\end{align*}
Here we can use the fact that $\nabla K\ast \1\in L^\infty$ as well as $ K\ast \1\in W^{1,1}_{loc}$ in order to deduce that
\[
\lim_{\epsilon\to0}\frac{K\ast \1(x+\epsilon\zeta(x))-K\ast \1(x)}{\epsilon}
=\nabla K\ast \1(x)\cdot\zeta(x),\quad {a.e.}.
\]
Moreover
\[
\left|\frac{K\ast \1(x+\epsilon\zeta(x))-K\ast \1(x)}{\epsilon}\right|
\le \sup_{t\in[0,1]}|\nabla K\ast \1(x+t\epsilon\zeta(x))|\,|\zeta(x)|
\le \|\nabla K\ast \1\|_{L^\infty}\,|\zeta(x)|, 
\]
thus by dominated convergence we obtain
\begin{align*}
    \lim_{\epsilon\to 0}&\int_{\Rnu} \left( \frac{K\ast \1(x+\epsilon\zeta(x))-K\ast \1(x)}{\epsilon}\right) \rho(x)\,dx=\int_{\Rnu}  \nabla K\ast \1(x)\cdot\zeta(x)\rho(x)\,dx.
\end{align*}

Combining all the terms in \eqref{1}, dividing by $\epsilon$ and then sending $\epsilon \to 0^+$ we get
\begin{align*}
    0\le& -\frac{1}{\tau}\int_\Rnu (x-T(x))\cdot \zeta(T(x))\rho_0(x)\, dx \\
    &+ \int_\Rnu \left(\Phi(\rho)-\Phi'\left(\rho(x)\right)\rho(x)\right)|\nabla\cdot\zeta(x)| \,dx \\
    &-\frac{\chi}{2} \int_\Rnu \rho(x) \int_\Rnu \nabla K(x-y)\cdot (\zeta(x)-\zeta(y))\rho(y)\, dy\, dx \\
    &-\chi \int_\Rnu \nabla K \ast \1(x)\cdot \zeta(x)\rho(x)\,dx. 
\end{align*}
Performing the above calculations with $-\epsilon<0,$ we obtain the reverse inequality and hence,
\begin{align*}
    \frac{1}{\tau}\int_\Rnu (x-T(x))\cdot \zeta(T(x))\rho_0(x) dx =&   \int_\Rnu \left(\Phi(\rho(x))-\Phi'\left(\rho(x)\right)\rho(x)\right)\nabla\cdot\zeta(x) dx \\
    &\quad-\frac{\chi}{2} \int_\Rnu \rho(x) \int_\Rnu \nabla K(x-y)\cdot (\zeta(x)-\zeta(y))\rho(y) dy dx \\
    &\quad-\chi \int_\Rnu \nabla K \ast \1(x)\cdot \zeta(x)\rho(x)dx. 
\end{align*}
Let $\varphi\in C^\infty_c(\Rnu)$ such that $\zeta=\nabla\varphi.$ We do a Taylor expansion of $\varphi$ around $T(x)$ and use the definition of push-forward measure and the H\"{o}lder continuity estimate to get
\begin{equation*}
    \frac{1}{\tau}\int_\Rnu (x-T(x))\cdot \nabla\varphi(T(x))\rho_0(x) \,dx =\frac{1}{\tau} \int_\Rnu \varphi(x)\left(\rho_0(x)-\rho(x)\right)\,dx + o(\tau).
\end{equation*}
On the other hand, an integration by parts on the diffusion term gives
\[
\int_\Rnu \left(\Phi(\rho(x))-\Phi'\left(\rho(x)\right)\rho(x)\right)\Delta \varphi(x)\, dx = \int_\Rnu \rho(x) \nabla \Phi'(\rho(x))\cdot \nabla \varphi(x) \,dx
\]

Replacing $\rho_0$ by $\rho_\tau^n$ and $\rho$ by $\rho_\tau^{n+1},$ in the two previous equations, we get
\begin{align*}
    \frac{1}{\tau}\int_\Rnu \varphi(x) \left(\rho_\tau^n(x)-\rho_\tau^{n+1}(x)\right)\,dx+&o(\tau) =  \int_\Rnu \rho^{n+1}_\tau(x) \nabla \Phi'(\rho^{n+1}_\tau(x)) \cdot \nabla\varphi(x)\, dx  \\
    &-\frac{\chi}{2} \int_\Rnu \rho_\tau^{n+1}(x) \int_\Rnu \nabla K(x-y)\cdot (\nabla\varphi(x)-\nabla\varphi(y))\rho_\tau^{n+1}(y)\, dy\, dx \\
    &-\chi \int_\Rnu \nabla K \ast \mathbb{1}_{S^{k}_\tau}(x)\cdot \nabla\varphi(x)\rho_\tau^{n+1}(x)\,dx.  
\end{align*}
Let $0\leq t < s\leq T$ be fixed and set
\[
h= \left[\frac{t}{\tau}\right]+ 1,\,\mbox{ and }\,k= \left[\frac{s}{\tau}\right].
\]
Summing over the indices, we have 
\begin{align*}
    \int_\Rnu \varphi \rho_\tau^{k+1}(x)\,dx&-\int_\Rnu \varphi\rho_\tau^h(x) \, dx+o(\tau) \\
    =& -\tau \sum_{n=h}^{k} \int_\Rnu \rho^{n+1}_\tau(x) \nabla \Phi'(\rho^{n+1}_\tau(x)) \cdot \nabla\varphi(x) \, dx  \\
    &-\frac{\chi}{2} \tau\sum_{n=h}^{k}\int_\Rnu \rho_\tau^{n+1}(x) \int_\Rnu \nabla K(x-y)\cdot (\nabla\varphi(x)-\nabla\varphi(y))\rho_\tau^{n+1}(y)\, dy \, dx \\
    &-\chi\tau\sum_{n=h}^{k} \int_\Rnu \nabla K \ast \mathbb{1}_{S^{k}_\tau}(x)\cdot \nabla\varphi(x)\rho_\tau^{n+1}(x)\,dx,
\end{align*}
that can be rephrased in terms of piecewise constant interpolations
\begin{align*}
    \int_\Rnu \varphi \rho_\tau(s,x)dx&-\int_\Rnu \varphi\rho_\tau(t,x) \,dx+o(\tau) \\
    =& -\int_t^s \int_\Rnu \rho_\tau(\sigma, x) \nabla \Phi'(\rho_\tau(\sigma, x)) \cdot \nabla\varphi(x)\, dx \, d\sigma\\
    &-\frac{\chi}{2}\int_t^s\int_\Rnu \rho_\tau(\sigma,x) \int_\Rnu \nabla K(x-y)\cdot (\nabla\varphi(x)-\nabla\varphi(y))\rho_\tau(\sigma,y) \, dy \, dx \, d\sigma \\
    &-\chi\int_t^s \int_\Rnu \nabla K \ast d_\tau(\sigma,x)\cdot \nabla\varphi(x)\rho_\tau(\sigma,x)\, dx \, d\sigma,
\end{align*}
We are now ready to pass to the $\tau\to 0$ limit. We start by noting that the first term on the RHS can be passed to the limit since
\[
\int_t^s\int_\Rnu \rho_\tau \,\nabla\Phi'\left( \rho_\tau \right) \cdot \nabla \varphi \, dx \, d\sigma = \int_t^s\int_\Rnu \left( f\left(\rho_\tau^\frac{\gamma}{2}\right) - \tfrac{\gamma}{2} f'\left(\rho_\tau^\frac{\gamma}{2}\right) \rho_\tau^\frac{\gamma}{2} \right) \Delta \varphi \, dx \, d\sigma,
\]
thanks to the strong convergence of $\rho_\tau^{\frac{\gamma}{2}}$ in $L^2([0,T]; \Rnu)$ and assumptions on $f$.
Moreover, using the weak convergence (up to a subsequence) of $\nabla \rho_\tau^{\frac{\gamma}{2}}$ and an integration by parts we can conclude that
\begin{equation*}
    \int_t^s \int_\Rnu \left(f\left(\rho_\tau^{\frac{\gamma}{2}}\right)-\frac{\gamma}{2} f'\left(\rho_\tau^{\frac{\gamma}{2}}\right)\rho_\tau^{\frac{\gamma}{2}}\right)\Delta\varphi \,dx \,d\sigma \to \int_t^s \int_\Rnu \rho\nabla\Phi'\left(\rho\right)\nabla\varphi\, dx \,d\sigma 
\end{equation*}
as $\tau\to 0$. Results in Propositions \ref{narrow_conv} and \ref{d_conv} allow to pass to the limit in the remaining terms in order to have the desired weak formulation for $\rho$. Concerning $\beta$, by passing to the limit in \eqref{eq:perweak} we have
\begin{equation*}
    \int_t^s\int_\Rnu h(\rho(\sigma,x))(1-\beta(\sigma,x))\psi(\sigma,x)\,dx\,dt=0,
\end{equation*}
for all $\psi\in C_c^\infty([0,T]\times\Rnu)$.
\end{proof}

\section{Including the reaction term: The Splitting Scheme}\label{sec:reaction}
In this section we consider the main system \eqref{maineq}, that we recall below for convenience 
\begin{equation*}
\begin{cases}
     \partial_t\rho = \nabla\cdot \left( \rho\nabla\Phi'(\rho) -\chi\rho\nabla K\ast (\rho + \beta)\right)+M(\rho), & x\in\Rnu, t>0,\\
     h(\rho)\left(1-\beta\right)=0, & x\in\Rnu, t>0.
\end{cases}
\end{equation*}
where the reaction term $M(\rho)$ satisfies the properties in Assumption \ref{ass}. 

Let $T>0$ and fix a time step $\tau > 0$ such that $N\tau =T$ for some $N\in \mathbb{N},$ and a partition of $[0,T]$ into intervals $[t^n,t^{n+1}]$ with $t^n=n\tau$, $n\in\{0,\cdots,N-1\}.$ Given an initial condition $\rho^0,$ we split the time-integration of \eqref{maineq} on the interval $[t^n,t^{n+1}]$ into two parts: the transport step and the reaction step. Fix $\rho_\tau^0=\rho^0$, then for any $n\in\{0,\cdots,N-1\}$,
\begin{itemize}
 \item Transport step: Given $\rho_\tau^n$ from the reaction step (or the initial datum in the first step), solve the minimization problem
    \begin{equation}\label{a-d}
        \rho^{n+\frac{1}{2}} \in \argmin_{\rho\in\mpdm{m}}\left\{\frac{1}{2\tau} W_2^2\left(\rho,\rho_\tau^n\right)+\mF\left[\rho|\rho_\tau^n\right]\right\}
    \end{equation}
 in $[t^n,t^{n+1}]$ where $m=\|\rho_\tau^n\|_{L^1(\Rnu)}$ and $\mF$ is defined in \eqref{eq:im_ex_fun}. In this step, we also define $\beta^{n+1}(x)=\mathbb{1}_{S^{n+\frac{1}{2}}}(x)$.
    \item Reaction Step: Given $\rho^{n+\frac{1}{2}}\in L^\gamma([0,T]\times\Rnu)$ from the trasport step and the corresponding $\beta^{n+\frac{1}{2}}\in L^\infty([0,T]\times\Rnu)$, we solve 
    \begin{equation}  \label{reac}  
    \begin{cases}
        \partial_t \rho = M(\rho),\\
        \rho(t^n) = \rho^{n+\frac{1}{2}},
    \end{cases}
    \end{equation}
    in $[t^n,t^{n+1}]$ and for all $x\in\Rnu$. Then we set $\rho_\tau^{n+1}(x)=\rho(t^{n+1},x)$.
\end{itemize}

Let $(\rho_\tau^n,\beta_\tau^n)_{1\le n\le N}$ be the sequence obtained from the splitting scheme. Then we define the piecewise constant interpolations by
\begin{equation}\label{interpolation_splitting}
\rho_\tau(t,x)=\rho_\tau^n(x),\quad   \beta_\tau(t,x)=\beta_\tau^n(x), 
\end{equation}
for all $t\in[t^n,t^{n+1}]$ and $x\in\Rnu$. Note that $\beta_\tau$ only evolves in the Transport step, so it brings naturally the estimates proved in the previous Section.

Since the function $M$ is Lipschitz in $\rho$ we have local in time existence and uniqueness. By comparison, the growth condition in Assumption \ref{ass} ensures global well-posedness.

\begin{prop}[Estimates on the reaction steps]\label{reac_est}
Let $(\rho_\tau^n,\beta_\tau^n)_{n\in \mathbb{N}}$ be the sequence obtained from the splitting scheme. Then we have
\begin{align*}
    \|\rho_\tau^{n+1}\|_{L^\gamma(\Rnu)}\leq  \|\rho^{n+\frac{1}{2}}\|_{L^\gamma(\Rnu)}e^{C\tau},\, &    \mathsf{m}_2[\rho_\tau^{n+1}]\le e^{C\tau} \mathsf{m}_2[\rho^{n+\frac{1}{2}}]\\ \mbox{and }\|\nabla\left(\rho_\tau^{n+1}\right)^{\frac{\gamma}{2}}\|_{L^2(\Rnu)}&\leq  e^{C\tau}\|\nabla\left(\rho^{n+\frac{1}{2}}\right)^{\frac{\gamma}{2}}\|_{L^2(\Rnu)}.
\end{align*}
for some constant $C>0$ independent of $n$. 
\end{prop}
\begin{proof}
By multiplying the equation in \eqref{reac} by $\rho^{\gamma-1}$ and integrating over $\Rnu$ we have
\begin{align*}
    \frac{1}{\gamma}\frac{d}{dt}\int_\Rnu \rho^\gamma(t,x) dx &= \int_\Rnu \rho^{\gamma-1}(t,x)M(\rho(t,x)) dx\\
    &\le C_M \int_\Rnu \rho^\gamma(t,x) dx.
\end{align*}
Using Gr\"{o}nwall Lemma in the time-interval $[t^n,t^{n+1}],$ we obtain
\[ \|\rho_\tau^{n+1}\|_{L^\gamma(\Rnu)}\leq  \|\rho^{n+\frac{1}{2}}\|_{L^\gamma(\Rnu)}e^{C\tau}.\]
Regarding the second moment of $\rho^{n+\frac{1}{2}},$ we note that 
\begin{align*}
    \frac{d}{dt}\int_\Rnu x^2\rho(t,x)dx = \int_\Rnu x^2 M(\rho(t,x)) dx\le C_M\int_\Rnu x^2 \rho(t,x) dx.
\end{align*}
Integratating in time in the interval $[t^n,t^{n+1}]$ gives the required bound.
Moreover, using the linear growth control on $M$ we have
\begin{align*}
    \frac{1}{\gamma}\frac{d}{dt}\int_\Rnu |\nabla\rho^{\frac{\gamma}{2}}|^2 dx &= \int_\Rnu \nabla\rho^{\frac{\gamma}{2}}\cdot \nabla\left(\rho^{\frac{\gamma}{2}-1}M(\rho)\right) dx\\
    &= \int_\Rnu \nabla\rho^{\frac{\gamma}{2}}\cdot \nabla\rho^{\frac{\gamma}{2}-1}M(\rho) dx+\int_\Rnu \rho^{\frac{\gamma}{2}-1}M'(\rho)\nabla\rho^{\frac{\gamma}{2}}\cdot \nabla \rho dx\\
    &\le C(\gamma,M) \int_\Rnu |\nabla\rho^{\frac{\gamma}{2}}|^2 dx.
\end{align*}
Again, by doing a Gr\"{o}nwall estimate in the time interval $[t^n,t^{n+1}]$ we get the desired estimate.
\end{proof}

The previous computation also yields finiteness of the mass of $\rho^{n+\frac{1}{2}}$
\begin{equation}
    \|\rho_\tau^{n+1}\|_{L^1(\Rnu)}\leq  \|\rho^{n+\frac{1}{2}}\|_{L^1(\Rnu)}e^{C\tau}<\infty
\end{equation}
for any $n\in\{0,\cdots,N\}$. 
As a consequence of the previous Proposition and the results in the previous section we have the following corollary.
\begin{cor}\label{rec_conv}
    Let $(\rho_\tau,\beta_\tau)$ defined in \eqref{interpolation_splitting} be the piecewise interpolation obtained from the splitting scheme. Then we have
\begin{align*}
    \sup_{t\in[0,T]}\|\rho_\tau\|_{L^\gamma(\Rnu)}\leq  C,\, &    \sup_{t\in[0,T]}\mathsf{m}_2[\rho_\tau]\le C\\ \mbox{and } \sup_{t\in[0,T]}\|\nabla\left(\rho_\tau\right)^{\frac{\gamma}{2}}&\|_{L^2(\Rnu)}\leq C,
\end{align*}
for some positive constant $C$ depending only on $T$ and the initial data and not on $\tau$.
\end{cor}

\begin{prop}\label{dBL_conv}
The piecewise constant interpolation $\rho_\tau$ defined in \eqref{interpolation_splitting} admits a subsequence converging uniformly to an absolutely continuous curve $\rho$ with respect to $d_\mathrm{BL}$ with values in $\mathcal{M}_{+}(\Rnu)$. Moreover, $\rho$ is a $d_\mathrm{BL}-$continuous function on [0,T].    
\end{prop}
\begin{proof}
We wish to bound the $d_\mathrm{BL}$ distance between two consecutive steps in the splitting scheme. From Proposition \ref{propdbl}, we note that
\begin{align*}
    d_\mathrm{BL}(\rho_\tau^n,\rho_\tau^{n+1})&\le W_1(\rho_\tau^n,\rho^{n+\frac{1}{2}})+ \|\rho^{n+\frac{1}{2}}-\rho_\tau^{n+1}\|_{L^1(\Rnu)}\\
    &\leq W_2(\rho_\tau^n,\rho^{n+\frac{1}{2}})+ \|\rho^{n+\frac{1}{2}}-\rho_\tau^{n+1}\|_{L^1(\Rnu)}.
\end{align*}
From the reaction step we have 
\begin{align*}
    \rho_\tau^{n+1}(x) &= \rho^{n+\frac{1}{2}}(x) + \int_{t^n}^{t^{n+1}} M(\rho(x,t)) \, dt.
\end{align*}
Integrating over $\Rnu$ we get
    \begin{align*}
    \|\rho_\tau^{n+1}- \rho^{n+\frac{1}{2}}\|_{L^1(\Rnu)} &\leq  C_M\int_{t^n}^{t^{n+1}}\int_\Rnu  \rho(t,x) dt dx\\
   &\le \tau C(T,\|\rho^0\|_{L^1}).
\end{align*}
Let $0 \leq s<t\leq T$ be two time instances and $m,n \in \mathbb{N}$ with $m\leq n - 1$, be such that
\[
s\in ((m - 1)\tau ,m\tau ]\, \mbox{ and }\, t\in ((k - 1)\tau ,k\tau ].
\]
Using the properties of the $d_\mathrm{BL}$ distance we estimate
\begin{align*}
    d_\mathrm{BL}(\rho_\tau^m,\rho_\tau^n)&\le \sum_{k=m}^{n-1}d_\mathrm{BL}(\rho_\tau^k,\rho_\tau^{k+1})\\
    & \le \sum_{k=m}^{n-1} \|\rho_\tau^{k+1}- \rho^{k+\frac{1}{2}}\|_{L^1(\Rnu)}+ \sum_{k=m}^{n-1} W_2(\rho_\tau^k,\rho^{k+\frac{1}{2}})\\
    & \le C(T,\|\rho^0\|_{L^1}) (n-m)\tau + C(\rho^0,\chi)\left(\tau(n-m)\right)^\frac{1}{2},
\end{align*}
where the last inequality follows from \eqref{estim_W_2}. For the interpolating sequence we then have
\begin{align*}
    d_{\mathrm{BL}}(\rho_\tau(s),\rho_\tau(t))&\le \left(\sum_{k=m}^{n-1} d^2_\mathrm{BL}(\rho^n,\rho^{n+1})\right)^{\frac{1}{2}}|n-m|^{\frac{1}{2}}\\
    &\le \left(2\sum_{k=m}^{n-1} \|\rho_\tau^{k+1}- \rho^{k+\frac{1}{2}}\|_{L^1(\Rnu)}+ 2\sum_{k=m}^{n-1} W_2(\rho_\tau^k,\rho^{k+\frac{1}{2}})\right)^{\frac{1}{2}}|n-m|^{\frac{1}{2}}\\
    &\le C \left(\tau^2(n-m)+ \tau\right)^\frac{1}{2}|n-m|^{\frac{1}{2}}\\
    & \le C \left(|t-s|^2+|t-s|\right)^\frac{1}{2}\\
    & \le C \sqrt{|t-s|},
\end{align*}
where the constant $C$ only depends on the final time $T$ and the initial parameters. Uniform narrow convergence then follows from the refined version of Ascoli-Arzel\`{a}.
\end{proof}

\begin{prop}\label{final_conv}
      Let $(\rho_\tau, \beta_\tau)_{\tau >0}$ be the interpolating sequence defined in \eqref{interpolation_splitting}. Then there exist two functions $\rho\in L^\gamma(0,T; L^\gamma(\Rnu))$ and $\beta\in L^\infty([0,T]\times \Rnu)$ and subsequences, again denoted by $(\rho_\tau )_{\tau >0}$ and $(\beta_\tau )_{\tau >0}$, such that
      \begin{equation*}
          \rho_\tau \to\rho\,\mbox{ strongly in }L^\gamma(0,T; L^\gamma(\Rnu))\, \mbox{ and }\beta_\tau \to\beta\,\mbox{ weak$-\ast$ in }L^\infty([0,T]\times \Rnu),
      \end{equation*}
      as $\tau\to 0$.
\end{prop}
  \begin{proof}
   We note that $\beta_\tau$ only evolves in the transport step and it is constant in the reaction step. The convergence of $\beta_\tau$ to $\beta$ then follows from Proposition \ref{d_conv}. Concerning $\rho_\tau$, by Proposition \ref{dBL_conv} there exists a (relabelled) subsequence $\rho_\tau$, such that $\rho_\tau (t) \rightharpoonup \rho(t)$ in $\mathcal{M}_{+}(\Rnu)$, for all $t\in [0,T]$. Then we can combine the results in Lemma \ref{regularity} and Proposition \ref{reac_est} in order to deduce the pointwise-in-time convergence in $L^\gamma(\Rnu)$. By recalling Proposition \ref{prop:conv} and Corollary \ref{rec_conv} we can argue the required convergence in $L^\gamma(0,T; L^\gamma(\Rnu))$.
  \end{proof}
We are now in the position for proving the main result of the paper
\begin{proof}[Proof of Theorem \ref{mainth}]
It will be enough to identify the limit couple $(\rho,\beta)$ obtained in Proposition \ref{final_conv} as a weak solution in the sense of Definition \ref{def:weak_sol}. Regarding $\beta$, since there is no contribution from the reaction step, we procced as in Proposition \ref{cons_JKO}. In order to check the weak formulation for $\rho$, we are only required to combine the steps in Proposition \ref{cons_JKO} together with a similar reconstruction for the reaction term. Let $\varphi\in C_c^\infty(\Rnu)$ and test \eqref{reac} against $\varphi$
\begin{equation}\label{disc_reac}
\begin{aligned}
     \int_\Rnu\frac{\rho_\tau^{n+1}(x)- \rho^{n+\frac{1}{2}}(x)}{\tau}\varphi(x)\,dx = & \int_\Rnu\int_{t^n}^{t^{n+1}} \partial_t\rho(x,\sigma)\varphi(x)\,dx\,d\sigma+o(\tau)\\
     &=\int_\Rnu\int_{t^n}^{t^{n+1}} M(\rho(x,\sigma))\varphi(x)\,dx\,d\sigma+o(\tau).
\end{aligned}
\end{equation}
As already proven in Proposition \ref{cons_JKO}, after the transport step we have
\begin{equation}\label{disc_tra}
\begin{aligned}
    \frac{1}{\tau}\int_\Rnu \varphi(x) \left(\rho^{n+\frac{1}{2}}-\rho_\tau^{n+1}\right)dx&+o(\tau) =  \int_\Rnu \rho^{n+1}_\tau(x) \nabla \Phi'(\rho^{n+1}_\tau(x)) \cdot \nabla\varphi(x) dx  \\
    &-\frac{\chi}{2} \int_\Rnu \rho_\tau^{n+1}(x) \int_\Rnu \nabla K(x-y)\cdot (\nabla\varphi(x)-\nabla\varphi(y))\rho_\tau^{n+1}(y) \, dy \, dx \\
    &-\chi \int_\Rnu \nabla K \ast \mathbb{1}_{S^{n}_\tau}(x)\cdot \nabla\varphi(x)\rho_\tau^{n+1}(x)\, dx.  
\end{aligned}
\end{equation}
For any $0 < t < s < T$, there are two uniquely determined integers $m,k$ such
that $t\in (m\tau ,(m+ 1)\tau ]$ and $s\in (k\tau ,(k+ 1)\tau ]$. Summing \eqref{disc_reac} and \eqref{disc_tra}, multiplying by $\tau$ and summing over $n$ ranging from $m$ to $k - 1$, we obtain
\begin{align*}
    \int_\Rnu \varphi(x) \rho_\tau(s,x)dx&-\int_\Rnu \varphi(x)\rho_\tau(t,x)dx+o(\tau) \\
    =& -\int_t^s \int_\Rnu \rho_\tau(\sigma, x) \nabla \Phi'(\rho_\tau(\sigma, x)) \cdot \nabla\varphi(x) dx \, d\sigma\\
    &-\frac{\chi}{2}\int_t^s\int_\Rnu \rho_\tau(\sigma,x) \int_\Rnu \nabla K(x-y)\cdot (\nabla\varphi(x)-\nabla\varphi(y))\rho_\tau(\sigma,y) dy dx d\sigma \\
    &-\chi\int_t^s \int_\Rnu \nabla K \ast \beta_\tau(\sigma,x)\cdot \nabla\varphi(x)\rho_\tau(\sigma,x)dx d\sigma\\
    &+\int_t^s \int_\Rnu M(\rho_\tau(x,\sigma))\varphi(x)\,dx\,d\sigma.
\end{align*}
The strong convergence obtained in Proposition \eqref{final_conv} as well as the assumptions on $M$ allow us to pass to the limit as $\tau \to 0$ in all the terms in order to  obtain the weak formulation
for $\rho$. 
\end{proof}

\section{Conclusion and perspectives}\label{sec:conclusion}
In this paper, we provide an existence result for weak solutions to the constrained non-linear non-local PDE in \eqref{maineq}, which is motivated by the pseudo-stationary limit $\tau\to 0^+$ in the multi-species model in \eqref{system} for multiple sclerosis. The proof relies on a variational splitting scheme isolating transport and reaction contribution in the equation. Our motivation in studying \eqref{maineq} comes from the study of stationary states of system \eqref{system}, that exhibits pattern formation, characteristic of the disease. Studying the long-term behaviour of the constrained system \eqref{maineq} can be a further step in this direction. On the other hand, the numerical discretization of both \eqref{maineq} and \eqref{system} and their convergence is also of interest, in particular appropriate schemes that exploit the variational structure can allow to closely compute the stationary states and long-time asymptotic.

\section*{Acknowledgements}
The authors gratefully acknowledge O. Tse for his valuable suggestions and helpful discussions during a crucial stage of this work. The authors acknowledge useful discussions on this topic with M. Di Francesco, A. Esposito and E. Radici.  The research of the authors is supported by the InterMaths Network, \url{www.intermaths.eu} and by the INdAM project N.E53C25002010001 ``Modelli di reazione-diffusione-trasporto: dall'analisi alle applicazioni''. SF is partially supported by the Italian “National Centre for HPC, Big Data and Quantum Computing” - Spoke 5 “Environment and Natural Disasters” and by the INdAM project N.E5324001950001 ``Teoria e applicazioni dei modelli evolutivi:
trasporto ottimo, metodi variazionali e
approssimazioni particellari deterministiche''. 
\appendix
\section{Useful well-known results}\label{sec:appendix}
This section makes a list of well-known theorems and results that are going to be useful.
\begin{thm}\label{ascoli}[A refined version of Ascoli-Arzel\`{a} Theorem \cite{ambrosio2008gradient}]
    Let $(\mathcal{S}, d)$ be a complete metric space and let $T>0.$ Let $K\subset \mathcal{S}$ be a sequentially compact set w.r.t. a weaker topology $\sigma$ on $\mathcal{S},$ and let $u_n:[0,T]\to \mathcal{S}$ be curves such that
    $$
    u_n(t)\in K \quad \forall n\in\N, \quad t\in[0,T], 
    $$
    $$
    \limsup_{n\to \infty} d(u_n(s),u_n(t))\le \omega(s,t)\quad \forall s,t\in[0,T], 
    $$
    for a (symmetric) function $\omega:[0,T]\times[0,T]\to [0,\infty),$ such that 
    $$
    \lim_{(s,t)\to(r,r)}\omega(s,t)=0 \quad \forall r\in[0,T]\setminus \mathcal{C},
    $$
    where $\mathcal{C}$ is an (at most) countable subset of $[0,T].$ Then there exists an increasing subsequence $k\mapsto n(k)$  and a limit curve $u:[0,T]\to \mathcal{S}$ such that $$u_{n(k)}(t)\xrightarrow{\sigma} u(t)\quad \forall t\in[0,T],\quad u \quad \text{is d-continuous in}\quad  [0,T]\setminus \mathcal{C}. $$
\end{thm}

    

\begin{thm}\label{aubin}[Extended Aubin-Lions Lemma]
On a Banach space $(X,d)$, let  $\mathscr{Y}:X\to [0,\infty]$ be a given lower semi-continuous functional with relatively compact sub-levels in $X$. Let $\dd:X\times X\to [0,\infty]$ be a given pseudo-distance on $X$, that is, $\dd$ is lower semi-continuous and $\dd(\rho,\eta)=0$ for any $\rho,\eta\in X$ with $\mathscr{Y}[\rho],\mathscr{Y}[\eta]<\infty$ implies $\rho=\eta$.
    Let further $U$ be a set of measurable functions $u:[0,T]\to X,$ with a fixed $T>0.$ If 
    $$
    \sup_{u\in U}\int_0^T \mathscr{Y}[u(t)]dt<\infty \quad \text{and}\quad \lim_{h\downarrow 0}\sup_{u\in U}\int_0^{T-h} \dd(u(t+h),u(t))dt = 0,
    $$
    $U$ contains an infinite sequence $\{u_n\}_{n\in\N}$ that converges in measure (w.r.t. $t\in[0,T]$) to a limit $u:[0,T]\to X.$

\end{thm}

\begin{defn}
    A semigroup $G_\Psi:[0,+\infty]\times \mpd\to \mpd$ is a k-flow for a functional $\Psi:\mpd\to \R\cup\{+\infty\}$ with respect to the Wasserstein$-2$ distance $W_2$ if, for any arbitrary $\rho\in\mpd,$ the curve $s\mapsto G_\Psi^s\rho$ is absolutely continuous on $[0,+\infty]$ and satisfies the evolution variational inequality
    $$
    \frac{1}{2}\frac{d^+}{d\sigma}W_2^2(G_\Psi^\sigma\rho,\Tilde{\rho})|_{\sigma=s}+\frac{k}{2}W_2^2(G_\Psi^s\rho,\Tilde{\rho} \le \Psi(\Tilde{\rho})-\Psi(G_\Psi^\sigma\rho),
    $$
    for all $s>0,$ and for any $\Tilde{\rho}\in\mpd,$ such that $\Psi(\Tilde{\rho})<\infty.$
\end{defn}
The symbol $\frac{d^+}{d\sigma}$ stands for the limit superior of the respective difference quotients and equals the derivative if the latter exists.

\begin{thm}
    Assume that a functional $\Psi:\mpd\to \R\cup\{+\infty\}$ is $\lambda-$convex (along geodesics), with a modulus of convexity $\lambda\in \R,$ that is, along every constant speed geodesic $\rho:[0,1]\to \mpd,$
    $$
    \Psi[\rho(t)]\le (1-t)\Psi[\rho(0)]+t\Psi[\rho(1)]-\frac{\lambda}{2}t(1-t)W_2^2(\rho(0),\rho(1))
    $$
    holds for every $t\in[0,1].$ Then $\Psi$ posses a uniquely determined $k-$flow, with some $k\le \lambda.$ Conversely, if a functional $\Psi$ possesses a $k-$flow, and if it is monotonically non-increasing along that flow, then $\Psi$ is $\lambda-$convex, with some $\lambda\ge k.$
\end{thm}
Below we state the flow interchange lemma, see \cite{matthes2009family} for more details.

\begin{lem}(Flow interchange)\label{flowinterlem}
    Let $\Psi:\mpd\to (-\infty,\infty]$ be a lower semi-continuous functional which possesses a $\kappa-$flow $\mathfrak{S}^\Psi.$ Define further the dissipation of $\F$ along $\mathfrak{S}^\Psi$ by 
    $$\mathfrak{D}^\Psi\F(\rho):=\limsup_{s\downarrow0}\frac{1}{s}\left(\F(\rho)-\F(\mathfrak{S}^\Psi_s\rho)\right)$$
    for every $\rho\in\mpd.$ If $\rho_\tau^{n-1},$ and $\rho_\tau^n$ are two consecutive steps of the minimizing movement scheme (\ref{JKO}) then 
    \begin{equation}\label{flowinterchange}
        \Psi(\rho_\tau^{n-1})-\Psi(\rho_\tau^n)\ge \tau \mathfrak{D}^\Psi\F(\rho_\tau^n)+\frac{\kappa}{2}W_2^2(\rho_\tau^n,\rho_\tau^{n-1}),
    \end{equation}
    In particular, if $\Psi(\rho_\tau^{n-1})<\infty$ then $\mathfrak{D}^\Psi\F(\rho_\tau^n)<\infty.$
\end{lem}
\begin{cor}\label{flowintercor}
    Under the hypothesis of Lemma \ref{flowinterlem}, let the $\kappa-$flow $\mathfrak{S}^\Psi$ be such that for every $n\in\N,$ the curves $s\mapsto \mathfrak{S}^\Psi_s\rho_\tau^n$ lies in $L^\gamma(\Rn),$ where it is differentiable for every $s>0$ and continuous at $s=0.$ Moreover, let $\mathfrak{R}:\mpd\to (-\infty,\infty]$ satisfy
    \begin{equation*}
        \liminf_{s\downarrow 0}\left(-\frac{\dd}{\dt} \F(\mathfrak{S}^\Psi_t \rho_\tau^n)\Big|_{t=s}\right)\ge \mathfrak{R}(\rho_\tau^n).
    \end{equation*}
    Then the following two estimates hold:
    $$\Psi(\rho_\tau^{n-1})-\Psi(\rho_\tau^n)\ge \tau\mathfrak{R}(\rho_\tau^n)+\frac{\kappa}{2}W_2^2(\rho_\tau^n,\rho_\tau^{n-1})\quad \text{for every }\quad n\in \N,$$
    $$\Psi(\rho_\tau^N)\le \Psi(\rho_0)-\tau\sum_{n=1}^N \mathfrak{R}(\rho_\tau^n)+\tau\max(0,-\kappa)\F(\rho_0)\quad \text{for every }\quad N\in\N.$$
\end{cor}

\end{document}